\documentclass{article}
\usepackage{arxiv}
\usepackage[utf8]{inputenc} % allow utf-8 input
\usepackage{hyperref}       % hyperlinks
\usepackage{url}            % simple URL typesetting
\usepackage{booktabs}       % professional-quality tables
\usepackage{amsfonts}       % blackboard math symbols
\usepackage{nicefrac}       % compact symbols for 1/2, etc.
\usepackage{microtype}      % microtypography
\usepackage{lipsum}

\usepackage{amsmath, amsfonts, amsthm,amssymb}
\usepackage[english]{babel}
\usepackage{graphicx}

\newtheorem{theorem}{Theorem}

\usepackage{xcolor}

\title{Dota Underlords game is NP-complete}

\author{
  Alexander A. ~Ponomarenko, Dmitry V. ~Sirotkin \\
  Laboratory of Algorithms and Technologies for Network Analysis\\
  National Research University Higher School of Economics \\
  Nizhny Novgorod, Russia\\
  \texttt{aponomarenko@hse.ru, dsirotkin@hse.ru} \\
}

\begin{document}
\maketitle

\begin{abstract}
In this paper, we demonstrate how the problem of the optimal team choice in the popular computer game Dota Underlords can be reduced to the problem of linear integer  programming. We propose a model and solve it for the real data. We also prove that this problem belongs to the NP-complete class and show that it reduces to the maximum edge weighted clique problem.

\end{abstract}

% keywords can be removed
%\keywords{Целочисленное программирование \and NP-трудность \and NP-полнота}

\section{Introduction}
People love to play games. Many games and puzzles that people play are interesting by its complexity: you need to be smart enough to solve it. In many cases, such complexity can be expressed as computation complexity depending on input size. For example, it has been shown \cite{fraenkel1981computing} that the chess game belongs to the EXPTIME class complexity; 
decision problem of player in legendary video game ``Tetris" is NP-hard \cite{breukelaar2004tetris}. It was shown that the puzzle ``Sokoban"  has polynomially solvable \cite{hearn2005pspace}.

The special place in theoretical computer science has NP-complete computational class.
It was found in \cite{kaye2000minesweeper} that ``Minesweeper" belongs to the NP-complete class. The problem of finding a minimal number of chip movements in a generalized version of 15-puzzle for the board of size $N \times N$ belongs to the NP-class also  \cite{ratner1986finding}.

In some sense, every NP-complete problem is a puzzle, and vice-verse, many puzzles are NP-complete. For a deeper study of the topic of the computational complexity of puzzles and games, we refer the reader to the review \cite{costa2018computational}.

In this paper, we consider a popular video game Dota Unlderlords. It is one of the so-called auto-chess games. It turns out that this problem can be represented as a combinatorial optimization problem, which belongs to the class NP-complete.

The article organized as follows. In section \ref{SectionDUDescription} Dota Underlords gameplay is described.
We present the formulation of the Dota Underlords problem as a linear integer programming problem in section \ref{SectionDUIP}. In the \ref{SectionNPCompleteProof} section, we show the NP-completeness of this task and reduce the problem to the maximum edge-weighted clique problem. The solution of the integer programming model for the real data is published in section \ref{SectionComputationalResults}. Traditionally, we summarize in the section \ref{SectionConclusion} ``Conclusion".

\section{Dota Underlords game play description}
\label{SectionDUDescription}

During the game, eight players build a team of ``heroes"-- creatures that can fight each other on the game map. Each of the heroes has basic parameters: health, damage, attack speed, and others, as well as a special ability that determines its role in the game. Each hero belongs to two or more ``alliances"-- sets that unite several heroes. For example, the hero Enchantress belongs simultaneously to the alliance ``druids" and to the alliance ``predators". When there are several heroes in the team who belong to the same alliance (for each alliance this number is individual), the player receives a bonus consisting of improving the characteristics of his heroes or worsening the characteristics of his enemy's heroes.

Also during the game, you can strengthen your heroes by upgrading them to higher levels or by purchasing in-game items. In this work, these aspects will not be taken into account.

Thus, the strength of a players team is determined by:

\begin{enumerate}
    \item The power of selected heroes
    \item Bonuses from the alliances which they are belong
\end{enumerate}

\section{Problem reformulation to the linear integer programming language }
\label{SectionDUIP}

\subsection{The simplest problem statement}

We formalize the problem as follows. We assume that in total we have $n$ heroes to choose from. We assume that the strength (power) of some $i$-th hero is presented by some non-negative value $s_i$. As $x_i$ we denote the belonging of hero $i$ to the team. Let $ x_i = 1 $, if the $ i $ -th hero belongs to the players team and $ x_i = 0 $ otherwise. The condition that there is no more than $m$ heroes in a team  can be written as $ \sum_{i=1}^n x_i \leq m $. Then in the simplest form this problem can be expressed as follows:

\begin{equation}
\begin{gathered}
    max \sum_{i=1}^n x_i s_i \\
    \sum_{i=1}^n x_i \leq m \\
    x_i \in \{0, 1\} \text{ – decision variable} \\
   n, m, s_i \text{ – constants}  \\
\end{gathered}
\end{equation}

In this statement, the problem is solved elementarily - the solution is to take $ n $ elements with the largest weights.

\subsection{Problem statement with alliances}
As mentioned, in``Dota Underlords" each hero belongs to two or more ``alliances" --- in turn, each alliance includes several heroes. When a team has several heroes from the same alliance (for each alliance this number is individual), the player receives a bonus, which is expressed in strengthening all the heroes from the alliance, strengthening all his heroes, or weakening all the heroes of the opponent. The last can be interpreted as a relative strengthening of the player’s heroes, and therefore only the first two cases will be considered throughout the work. It should be noted that for one alliance, there can be several bonuses that are unlocked by the different numbers of heroes of the corresponding alliance. These bonuses can also be of various types.

We propose to model this situation by introducing a 3-index tensor $ e_{ijk} \in \mathbb{R} $ which represents a bonus to the hero $i$ from the alliance $j$, in which there are at least $k$ heroes of the alliance $j$. In other words, $ e_{ijk} $ is the $k$-th bonus of the alliance $ j$ for the hero $i$.

Using the tensor $e_{ijk} $, we support both types of alliances -- those that give bonuses to their members and those that give bonuses to all the heroes of the player. Moreover, the alliances of the considered types differ only in one thing. In the alliances that give a bonus to their members, the value of $e_{ijk} $ is zero if and only if the $i$-th hero does not belong to the $j$-th alliance. In the general case, this is not necessarily true. Note that the tensor $ e_{ijk} $ is sparse for the real instances of auto-chess games since the alliances from which bonuses go to all the heroes are few.

We propose to control the occurrence of the bonus $ e_{ijk} $ in the total strength of the team using the control binary variable $ I_{ijk} $.
So we can write down the objective function as the following sum $ \sum_{i = 1}^{n} x_i s_i + \sum_{i=1}^{n} \sum_{j=1}^{t}  \sum_{k=1}^{q} e_{ijk} I_{ijk} $.
The connection between the variables $ x_{i} $ and $ I_{ijk} $ is given by the inequalities
$\forall{i,j,k} :  \sum_{i'=1}^{n} a_{i'j} x_{i'} - k \ge M( I_{ijk} - 1)$. 

These inequalities do not allow the binary variable $ I_{ijk} $ to take the value 1 if the solution includes less than $k$ heroes from the alliance $j$. When the solution contains less than $m$ heroes from the alliance $j$, the left side of this inequality is negative, so for the inequalities to be observed, the right side should be even smaller. It is possible only when the binary variable $ I_{ijk} $ is zero. In this case, the right-hand side is $-M $, where $M$ is a big constant known to be larger than $k$, that is, larger than the maximum size of the alliance $q$.

We require that the bonus for the hero $i$ can be activated ($ I_{ijk} = 1 $) only if the hero $i$ belongs to the solution. This is given by the inequalities $ \forall {i, j, k}: I_{ijk} \le x_i $. We also want the bonus $ e_{ijk} $ to be activated only if the character $ i $ belongs to the alliance $j$. For this, we include in the model inequalities $ \forall{i, j, k} :\, I_{ijk} \le a_{ij} $.

Thus, after introducing the alliances into the model, the system of equations can be written as the following:

\begin{equation}
\label{eq:DUIP}
\begin{gathered}
\textbf{Objective function}\\
max \sum_{i=1}^{n} x_i s_i + \sum_{i=1}^{n} \sum_{j=1}^{t}  \sum_{k=1}^{q} e_{ijk} I_{ijk} \\
\textbf{Constraints for the input data}\\
\forall{j} : \sum_{i=1}^n a_{ij} \le q \\
\textbf{Constraints for the decision variables} \\
\forall{i,j,k} :  \sum_{i'=1}^{n} a_{i'j} x_{i'} - k \ge M( I_{ijk}  - 1) \\
\sum_{i=1}^n x_i \le m   \\ 
% \sum_{i=0}^{n-1} a_{ij} x_{i} < k \\ 
\forall{i,j,k} :  I_{ijk}  \le x_i \\
%\forall{i,j,k} :  I_{ijk}  \le a_{ij} \\
\textbf{Decision variables} \\
I_{ijk} \in \{0, 1\} \text {, 1 – if for the hero } i \text{, the } k\text{-th bonus is activated for }  j \text{-th alliance,} \\
x_i  \in \{0, 1\} \text{, 1 -- if hero } i \text{ belongs to solution} \\
\textbf{Constants} \\
n \in \mathbb{N} \text{ -- number of heroues,} \\
m \in \mathbb{N} \text{ -- maximum size of the team}\\
t \in \mathbb{N} \text{ -- the total number of alliances} \\
q \in \mathbb{N} \text{ -- maximum size of an alliance,} \\
s_i  \in \mathbb{R} \text{ –- the strength of the hero } i, \\
e_{ijk} \in \mathbb{R} \text{ -- the bonus for the hero } i \text{,  if } k
\text{-th bonus is activated for the } j \text{-th alliance} \\
a_{ij} \in \{0, 1\} \text{ -- indicates if hero } i \text{ belongs to the alliance } j \\ 
\end{gathered}
\end{equation}

\section{Proof of an NP-completeness of the Dota Underlords problem}
\label{SectionNPCompleteProof}

To prove that a problem is NP-complete, it is necessary to show that it is both an NP-hard problem and that it belongs to the NP class. Let's us prove both statements.

\subsection{Reduction maximum density sub-graph problem to the Dota Underlords problem}

\begin{theorem}
\label{MEWC_DU}
The problem of finding the maximum dense subgraph of $k$ vertices reduces to the Underlords problem.
\end{theorem}
\begin{proof}

Consider its special case --- let all the alliances have a size equal to two, and the power of all the heroes is the same. Consider a special case of the Dota Underlords problem with the following restrictions:

\begin{enumerate}
    \item The power of all heroes is the same ($\forall i, j \; s_i=s_j$)
    \item Alliances can give bonuses only to the heroes that belong to the corresponding alliance. ($\forall i, j, k \; a_{ij}=0 \Longrightarrow e_{ijk} = 0$)
    
    \item All alliances have the same size equaling two ($\forall j \; \sum_i a_{ij}=2 $)
    \item All alliances give a bonus if and only if both heroes are present in the team ($\forall i, j \; e_{ij1}=0$)
    \item Bonuses from all alliances are the same ($\forall i, j, i', j' \; a_{ij}=1,\, a_{i' j'}=1 \Longrightarrow e_{ij2}=e_{i' j' 2}$)
\end{enumerate}

Then the data can be represented in the form of a graph $G(V, E) $, where the set of vertices $ V $ corresponds to the heroes, and the set of edges $E$ corresponds to the active alliances. You may notice that in this case, the optimal team of size $ k $ corresponds to the densest subgraph $ G' \subset G $ with $ k $ vertices. Density in this formulation can be understood as the value $ \frac{G'(E)}{G'(V)} $. Indeed, under these restrictions, the total strength of the team linearly depends on the number of active alliances, which corresponds to $ G'(E)$. Since $ k $ is invariable, with the increasing density of the graph $ G '$ the total strength of the team grows.

It was shown in \cite{downey1995fixed} that the problem of fining subgraph with a fixed size and maximum density is NP-complete. We have shown that it is a special case of the Dota Underlords problem, so it is reducible to Dota Underlords. Therefore the Dota Underlords problem is no less difficult than the well-known NP-complete problem. Thus the Dota Underlords problem is NP-hard.

\end{proof}

\subsection{Dota Underlords belongs to NP-class}
The decision version of Dota Underlords problem (problem with the answer ``yes" or ``no") can be formulated as follows: ``\textit{Is there a team with at most $ m $ heroes and with a total power greater than some given constant?}". Then, we are ready to state the following theorem.
\begin{theorem}
\label{DU_is_NP}
The decision version of Dota Underlords problem belongs to the  NP class.
\end{theorem}
\begin{proof}

By the definition of the NP-class, the problem belongs to the NP class, when the presented solution can be checked in polynomial time. In our case, the solution is a set of $m$ heroes. To verify the solution, we need to calculate the total power of the team. 

So, we just need to calculate the objective function. In turn, to do that, at first we need to find out what alliances are formed. This means that we need to calculate the number of non-zero elements for each column in the matrix $a_{ij}$, taking into account only rows corresponding to heroes from the team i.e  we mean submatrix $\{a_{ij}:  j \in \overline{1,m},\; i \in \{  i' \in \overline{1,n} :   x_{i'} = 1 \}  \}$. It is can be done for $O(nt)$ operations. 

After that, the objective function can be calculated in a straightforward way by $O(ntq)$ number of operations. Thus, we need can check the solution for the polynomial time on the input size. 
\end{proof}

\subsection{NP-completeness of Dota Underlords problem}
\begin{theorem}
	The Dota Underlords problem defined by the system of inequalities \eqref{eq:DUIP}  belongs to the class of NP-complete problems.
\end{theorem}

\begin{proof}
Theorem \ref{MEWC_DU} states that there exists a polynomial reduction of an NP-complete problem to DU. At the same time, by the theorem \ref{DU_is_NP} we showed that the problem DU belongs to the NP class. Thus, the DU problem is NP-hard, and at the same time, it lies in the NP class. Therefore the decision version of Dota Underlords problem is NP-complete.	
\end{proof}

\subsection{Reduction from the Dota Underlords problem to the maximum edge-weighted clique problem}
While working on the paper, we also found a reduction from the Dota Underlords problem to the well-known problem --- Maximum Edge Weighted Clique (MEWC). Thus, when solving individual instances of the Dota Underlords problem, anyone can use the already developed algorithms for the MEWC problem such as \cite{san2019new} or efficient quadratic formulations from \cite{hosseinian2017maximum}.

In this reduction, we will consider a problem with the maximum size of the alliance bounded by some constant $q$. The reduction will be proposed through the series of theorems, where each theorem describes the reduction from a less and less simplified version of DU to the maximum edge-weighted clique problem.

\begin{theorem}
\label{trivial_case}
    The Dota Underlords problem without alliances is reduced to the  maximum edge-weighted clique problem.
\end{theorem}
\begin{proof}
    
   We construct a graph $G$ with weighted edges such that the solution of the problem DU (Dota Underlords) follows from the solution of the problem MEWC (Maximum Edge-Weighted Clique). Moreover, the size of the MEWC problem is limited by a polynomial on the size of the DU problem.
    We construct the set $V^1 $ of $n$ vertices corresponding to the set of heroes in the DU problem. To each vertex we assign one of the heroes from the DU problem -- or, in other words, we name each vertex in honor of one of the heroes of the DU problem. We enumerate these vertices according to the order of the heroes $v_1^1$, $v_2^1$, ..., $ v_n^1$
    We additionally construct $m-1$ sets of vertices $ V^2 $, $ V^3 $, and so on up to $ V^m $, in each we also name one vertex in honor of one of the heroes of the DU problem. Similarly to the first set, we enumerate the vertices in the set $V^i $ as $ v_1^i $, $v_2^i$, ..., $v_n^i$.
    Denote the family of these sets as $ \mathcal{F} $. Thus, we get $m$ sets of $n$ vertices, where each set has one vertex corresponding to one of the heroes.
   We build edges in the graph as follows --  between the vertices $v_a^i $ and $v_{a'}^{i'} $ an edge is drawn if both of the following conditions are true:
    \begin{itemize}
        \item The vertices $v_a^i $ and $v_{a'}^{i'} $ correspond to different heroes ($a \neq a'$)
        \item The vertices $v_a^i $ and $v_{a'}^{i'} $ lie in different sets from the family $\mathcal{F}$ ($ i \neq i /'$)
    \end{itemize}

    Consider all the maximum clique in this graph. Obviously, in any such clique, there is exactly one vertex from each set $ V^i $ -- total $m$ vertices. Also, all these vertexes correspond to different heroes. Thus, each clique sets a team of heroes in the DU task. It should be noted that each team can correspond to several cliques.
  
    Now we introduce the heroes’ power. For this, we assign the weight $ \frac{s_a}{m-1} + \frac{s_{a '}}{m-1}$ to the edge connecting the vertices $s_a^i$ and $s_ {a'}^{i'}$. We show that the sum of the weights of the edges in a click corresponding to a certain team is exactly the strength of this team.
    Indeed, in a clique for each of its vertices, there is exactly $m-1$ edge incident to it. Then each term $\frac{s_i}{m-1}$ corresponding to a vertex with a subscript $i$ is included in the sum exactly $m-1$ times. It follows that the sum of all the weights of the edges in a clique is the sum of all the values $s_i$ corresponding to the numbers of the vertices that form this clique.      
\end{proof}

\begin{figure}[h!]
\begin{center}
\includegraphics[height=4in,width=3in,angle=0]{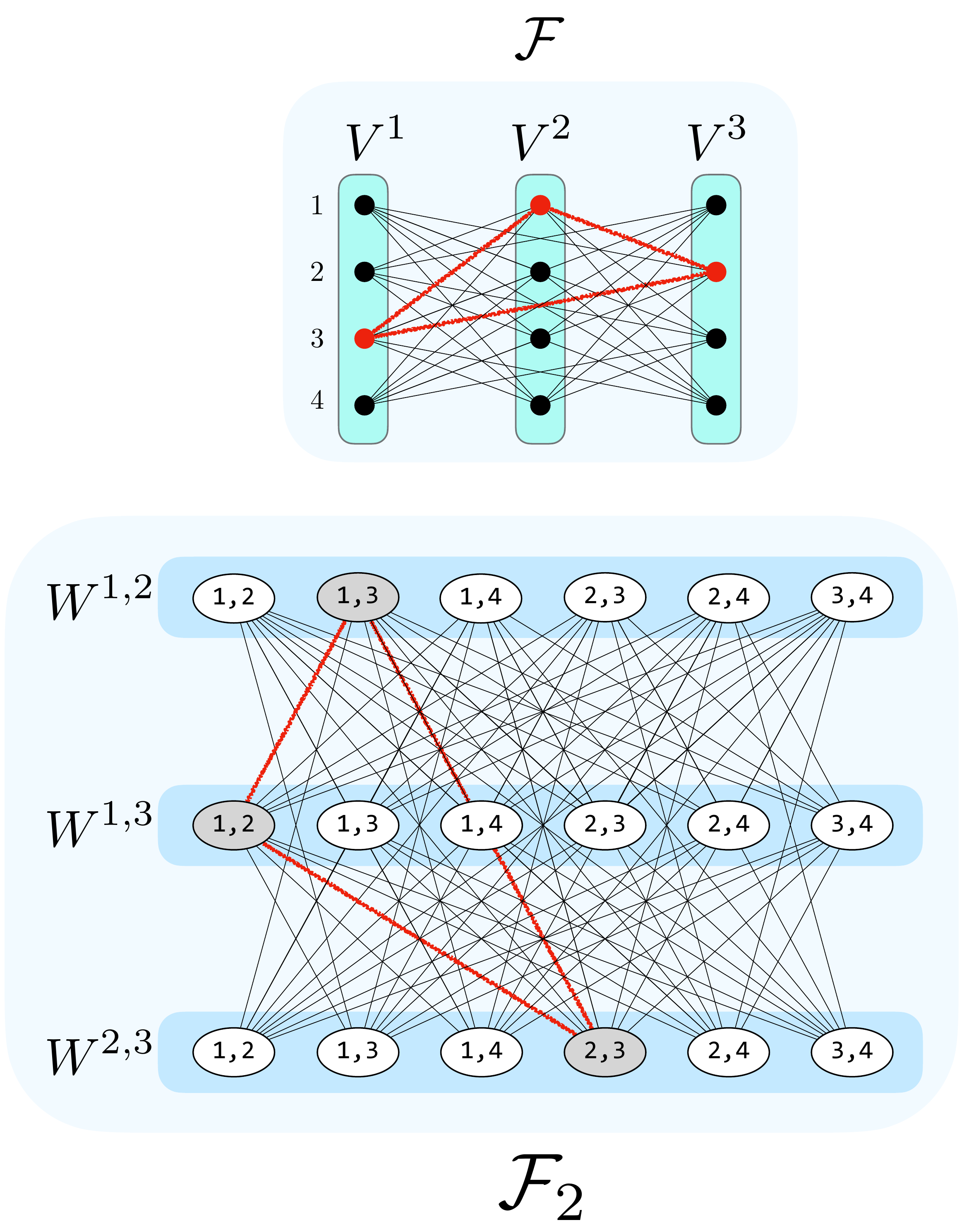}
\caption{An example of graph $G'$ from theorem \ref{second_case}.  Graph $G'$ consists of the sets $\mathcal{F}$ and $\mathcal{F}_2$ for the case $n=4$, $m=3$, $q=2$. The edges between every vertex of sets $\mathcal{F}$ and $\mathcal{F}_2$ are avoided for picture simplicity. The size of the maximal clique is 6. }
\label{fig:reduction}
\end{center}
\end{figure}

\begin{theorem}
\label{second_case}
    The Dota Underlords problem with alliances of the size 2 is reduced to the maximum edge-weighted clique.
\end{theorem}

\begin{proof}
We construct a graph $G'$ with weighted edges in such a way that the solution of the problem DU follows from the solution of the MEWC problem. Moreover, the size of the MEWC problem is limited by a polynomial on the size of the DU problem. We take a graph $G$ from the theorem \ref{trivial_case} as a base. We construct a set $ W^{1,2} $ with $ \binom{n}{2} $ vertices, where each vertex corresponds to an unordered pair of heroes. We enumerate these vertices in lexicographic order, respectively, by the order of the pairs $ w_{1,2}^1 $, $w_{1,3}^1 $,..., $w_{n-1, n}^1$. 

    We additionally construct $\binom{m}{2}-1$ vertex sets $W^{1,3} $, $W^{1,4} $ and so on up to $W^{m-1, m}$. In every of these set we also assign each vertex to an unordered pair of heroes from the DU problem. Similarly to the first set, we enumerate the vertices in the set $W^{i, j} $ as $ w_1^{i, j} $, $ w_2^{i,j} $, ..., $ w_{n-1, n }^{i, j} $.   
   We denote the family of these sets by $\mathcal{F}_2 $. Thus, we get $ \binom {m}{2} $ sets of $ \binom{n}{2}$ vertices, at the same time, vertexes at every set correspond to all possible pair of heroes.    
    In the graph we draw additional edges between the vertices $ w_{a, b}^{i, j} $ and $w_{a',b'}^{i',j'} $ an edge is drawn if following conditions are satisfied:
    \begin{itemize}
        \item The vertices $ w_{a, b}^{i, j} $ and $w_ {a', b'}^{i',j'} $ correspond to different pairs of heroes ($ a \neq a '\lor b \neq b'$)
        \item The vertices $ w_ {a,b}^{i, j} $ and $w_{a', b'} ^{i', j'} $ lie in different sets from the family $ F_2 $ ($ i \neq i '\lor j \neq j' $)
    \end{itemize}
    
    We also build all edges between all the vertices from the sets $ \mathcal {F} $ and $ \mathcal {F}_2 $. We assign a weight of 0 to all these new edges. Evidently, any maximal clique contains one vertex from each of the sets $V$ and $W$ of the families $ \mathcal {F} $ and $ \mathcal{F}_2 $. \\
    Assign to each edge  $ (v_a^k, w_{a,b}^{i, j}) $ or $ (v_b^k, w_{a,b}^{i,j}) $ some big constant weight $N$. These edges connect a vertex from the family $\mathcal{F}$ corresponding to a certain hero with a vertex from the family $ F_2 $ corresponding to a pair of heroes where this hero belongs. See the figure \ref{fig:reduction} for example of graph $G'$.

    Now we are going to show that any maximal clique in graph $G'$ containing a set of vertices from the family $\mathcal{F}$ corresponding to some set of heroes also contains the set of vertices from the family $\mathcal {F}_2$ corresponding to all pairs of said heroes from this team. In this clique, the edges connecting the vertices from the families $\mathcal{F}$ and $ \mathcal {F}_2$ make a total contribution to the weight equals to $2 \binom{n}{2} N$, because maximum clique includes exactly $2 \binom {n} {2}$ edges with additional big weight $N$ -- two incident to each vertex from $\mathcal {F}_2$.
    
    	It is easy to see that if a vertex in a clique belongs to the family $\mathcal{F}_2 $ and does not correspond to a pair of associated vertices of $ \mathcal {F} $ that are in the clique, then the clique will contain at least one edge that has an additional big constant weight $N$ less. Thus, the click will not have the maximum weight. Thus, the statement is proved.
    	
    	 Note that adding weights on the edges that are small compared to $ N $ preserves the truth of the statement. Since $ N $ is chosen arbitrarily, we can assume that all values of power and bonuses are small compared to $ N $. Therefore we add bonuses that an alliance of a pair of heroes with numbers $ a $ and $ b $ gives the hero with number $ c $ to the weights of the edges  $ (v_c^{k}, w_{a, b}^{i, j}) $ connecting the vertices from the sets $\mathcal {F}$ and $\mathcal {F}_2$. If the selected team has the heroes $a$, $b$, and $c$, then this bonus will be included in the weight of the clique. Since the same number of edges with an additional weight $ N$ are included in all the cliques under consideration, the maximum clique will be the one where the sum of the heroes' strengths (the sum of the edges’ weights between the vertices of the $\mathcal {F} $ family) and bonuses (the edges’ weights between the vertices of the families $\mathcal{F}$ and $\mathcal {F}_2$ without taking into account the constants $N$) is the maximum. Thus, the weight of the clique corresponds to the total bonus from the team, from which point the reduction is clear.
    
\end{proof}

%%%%%%%%%%%%%%%%%%%%%%%%%%%%%%%%%%%%%%%
% Теорема для случая альянса размером k
%%%%%%%%%%%%%%%%%%%%%%%%%%%%%%%%%%%%%%%

\begin{theorem}
\label{general_case}
    Dora Underlords problem with alliances of size $q$, reduces to the maximum edge-weighted clique.
\end{theorem}
	The proof will be constructed similarly to the proof of the theorem \ref {second_case}.
We will create additional vertices associated with all possible combinations of $q$ heroes. The bonuses from the formation of the corresponding alliances will be the same as in the theorem \ref {second_case} is on the edges.
The main difference will be that for the vertices we will use $q$ indices instead of two indices. The formal considerations are given below.
\begin{proof}
    
We construct a graph $G'$ with weighted edges in such a way that the solution of the problem DU follows from the solution of the MEWC problem. Moreover, the size of the MEWC problem is limited by a polynomial on the size of the DU problem.
    We take a graph $G$ from the theorem \ref{trivial_case} as the base. We construct a set $W^{1,q} $ with $\binom{n}{q} $ vertices, where each vertex corresponds to an unordered set of $q$ heroes. We enumerate these vertices in the lexicographic order corresponding to the order of  combinations of $ \binom {n}{q} $ elements of $ w^1_ {\underbrace {1,2, ..., q}_\text{total $ q $ indices }}, w^1_{\underbrace {1,2, ..., q + 1}_\text {total $ q $ indices}}, ..., w^1_ {\underbrace{n-q+1, ..., n-1, n}_\text{total $ q $ indices}} $. It is important that each of these elements has exactly $q$ indices.

We additionally construct $ \binom{m}{q}-1 $ vertex sets $ W^{1,2, ..., q} $, $W^{1,2, q + 1} $ and so on up to $ W^{n-q + 1, ..., n-1, n} $. In every of these sets we also assign each vertex to an unordered set of $q$ heroes of the DU problem according to $\binom{n}{q}$ possible combinations. Similarly to the first set, we enumerate the vertices in the set $V^{\overbrace {i, j, k, l, ...}^\text {total $ q $ indices}} $ as $w_{\underbrace {1,2, ..., q}_\text {total $ q $ indices}}^{\overbrace{i, j, k, l, ...}^\text{total $ q $ indices}} $, $w_{ \underbrace {1,2, ..., q + 1}_\text{total $ q $ indices}}^{\overbrace{i, j, k, l, ...} ^\text{total $ q $ indices}} $,  ..., $w_{\underbrace {n-q + 1, ..., n-1, n}_\text{total $ q $ indices}}^{\overbrace { i, j, k, l, ...}^\text {total $ q $ indices}} $.
    We denote the family of these sets as $ \mathcal {F} _q $. Thus, we get $ \binom{m}{q} $ sets with $ \binom{n}{q} $.
    We draw additional edges between the vertices $w_{\underbrace{a, b, c, ....}_\text{total $ q $ indices}}^{\overbrace{i, j, k , ...}^\text{total $ q $ indices}} $ and $w_{\underbrace{a',b', c', ...}_\text{total $ q $ indices}}^{\overbrace {i ', j', k ', ...}^\text{total $ q $ indices}}$. An edge is drawn if both of the following conditions are true:
    \begin{itemize}
        \item The vertices $w_{a, b, c, ...}^{i, j, k ...} $ and $ w_ {a', b', c', ...}^{i', j', k', ...}$ correspond to different set of heroes ($ a \neq a '\lor b \neq b' \lor c \neq c', ... $)
        \item The vertices $ w_{a, b, c, ...}^ {i, j, k ...} $ and $ w_{a', b', c',...}^{i', j', k', ...} $ belong to different sets from the family $ F_q $ ($ i \ neq i'\lor j \ neq j' \lor k \neq k', ... $)
    \end {itemize}
    
    We also build all the edges between every vertices of the sets $ F $ and $ F_q $. We assign a weight of 0 to all those edges. Indeed, any maximal clique contains one vertex from each of the sets $ V $ and $ W $ of the families $\mathcal {F}$ and $\mathcal {F}_q$. We assign to each edge with a shape $ (v_a^k, w_ {a, b, c, ...}^ {i, j, k, ...}) $ or $ (v_b^k, w_{a, b, c , ...}^{i, j, k, ...}) $ some big constant weight $ N $. Thus, every edge connect a vertex from the family $ \mathcal{F} $ corresponding to some hero with a vertex from the family $ F_q $ corresponding to the set of $ q $ heroes (a hero forms alliance with set of $q$ heroes). 
    
     Now we are going to show that any maximal clique $\mathcal{C}$ in graph $G'$ has a set of vertices from the family $\mathcal{F}$ corresponding to a certain set of heroes, and at the same time $\mathcal{C}$ contains a set of vertices from the family $ \mathcal{F}_q $ corresponding to all combinations of $q$ heroes. In the clique $\mathcal{C}$, the edges connecting the vertices from the families $ \mathcal{F} $ and $ \mathcal{F}_q $ make a total contribution to the weight of $\mathcal{C}$ and this contribution equals to $ q \binom{n}{q} N $. It is because the clique includes exactly $ q \binom{n}{q} $ edges with an additional big constant weight $N$ --- q incident to each vertex from $ \mathcal {F}_q $.  It is not hard to see that when a vertex in a clique belongs to the family $ \mathcal{F}_q $, and does not belong to a set of $q$ vertices from $ \mathcal {F} $, then the clique $\mathcal{C}$ contains at least one edge less with an additional big constant weight $ N $ among the edges of the clique. Thus, the clique $\mathcal{C}$ does not have the maximum weight. Thus, the statement is proved. 
    
    Note that adding weights on the edges that are small compared to $N$ preserves the truth of the statement. Since $N$ is chosen arbitrarily, we can assume that all the values of heroes’ power and bonuses are small compared to $N$.
	Therefore we add bonuses to the weights of the edges of shape $ (v_x^{y}, w_{a, b, c, ...} ^ {i, j, k, ...}) $ connecting the vertices from families of $\mathcal{F} $ and $ \mathcal{F}_q $. If the team has the heroes $a$, $b$, $c$, ... and hero $x$, then this bonus will be included in the weight of the maximum edge-weighted clique. Since all the cliques under consideration have the same number of edges with an additional weight $N$, the maximum clique will be the one where the maximum is the sum of the heroes' powers (the sum of the weights of the edges between the vertices of the $ \mathcal{F} $ family) and bonuses (the weights of the edges between the vertices of the families $ \mathcal{F} $ and $ \mathcal {F}_q $ without taking into account the constants $ N $).
    
Thus, the weight of the clique corresponds to the total bonus from the team, from where the reduction is clear.
    
\end{proof}

%\section{Практическое применение для реальной задачи Dota Underlords}
\section{Model application for real data}
\label{SectionComputationalResults}
  
We apply this model to analyze the real Dota Underlords problem. Note that our result should not be considered as some objective assessment of the quality of the team of heroes. The reason is the inevitable simplification of the heroes’ power as well as the influence that the alliances have. Each hero in Underlords has a certain ability, which is activated when various conditions satisfied, and in addition, the ability has some recharge time. Alliance abilities and bonuses are also very diverse in their influence on the game -- they can cause damage, heal allies, prevent enemies from using their abilities, and more. Fortunately, the game has a system of five ``tiers", arranged so that the characters inside the tier are approximately equal in strength.

Within the simplified model, we accept the following:
\begin{enumerate}
\item The forces of all the heroes of the first tier are equal to 1,  the second -- 2, the third -- 3, the fourth -- 4, the fifth -- 5;
\item Alliances give the same percentage bonus to everyone they equally influence;
\item The alliance bonus is approximately 10-30 percent of the hero’s power.
\end{enumerate}

Information about the strength of the heroes and the structure of alliances is given in the table \ref{table:aliances}. 
A complete table defining a matrix of bonuses from the alliances $e_{ijk} $ can be found in our \href{https://github.com/aponom84/UnderLords/blob/master/UnderLordsData.xlsx}{repository} \cite{UnderLordsInput}.

\begin{table}
\center
\resizebox{!}{9cm} {
\begin{tabular}{llrl}
{\#} &                 Heroes &  Power &                       Alliances \\
\midrule
0  &                 tusk &      1 &               savage, warrior  \\
1  &           venomancer &      1 &               scaled, summoner \\
2  &         shadow demon &      1 &               demon, heartless \\
3  &          drow ranger &      1 &    heartless, hunter, vigilant \\
4  &          bloodseeker &      1 &           blood-bound, deadeye \\
5  &         nyx assassin &      1 &               assassin, insect \\
6  &       crystal maiden &      1 &                    human, mage \\
7  &                 tiny &      1 &           primordial, warrior  \\
8  &             batrider &      1 &                  knight, troll \\
9  &               magnus &      1 &                  druid, savage \\
10 &             snapfire &      1 &                 brawny, dragon \\
11 &           arc warden &      1 &           primordial, summoner \\
12 &                razor &      1 &               mage, primordial \\
13 &               weaver &      1 &                 hunter, insect \\
14 &              warlock &      1 &  blood-bound, healer, warlock  \\
15 &               dazzle &      2 &                  healer, troll \\
16 &         earth spirit &      2 &               spirit, warrior  \\
17 &         storm spirit &      2 &                   mage, spirit \\
18 &         witch doctor &      2 &                troll, warlock  \\
19 &          bristleback &      2 &                 brawny, savage \\
20 &     legion commander &      2 &                champion, human \\
21 &        queen of pain &      2 &                assassin, demon \\
22 &     nature's prophet &      2 &                druid, summoner \\
23 &                 luna &      2 &               knight, vigilant \\
24 &           windranger &      2 &               hunter, vigilant \\
25 &            ogre magi &      2 &       blood-bound, brute, mage \\
26 &                pudge &      2 &            heartless, warrior  \\
27 &          beastmaster &      2 &                 brawny, hunter \\
28 &         chaos knight &      2 &                  demon, knight \\
29 &              slardar &      2 &               scaled, warrior  \\
30 &              abaddon &      3 &              heartless, knight \\
31 &                viper &      3 &               assassin, dragon \\
32 &           juggernaut &      3 &               brawny, warrior  \\
33 &         ember spirit &      3 &               assassin, spirit \\
34 &                   io &      3 &              druid, primordial \\
35 &         shadow fiend &      3 &                demon, warlock  \\
36 &                lycan &      3 &        human, savage, summoner \\
37 &          broodmother &      3 &               insect, warlock  \\
38 &            morphling &      3 &               mage, primordial \\
39 &          lifestealer &      3 &               brute, heartless \\
40 &           omniknight &      3 &          healer, human, knight \\
41 &          terrorblade &      3 &                  demon, hunter \\
42 &        shadow shaman &      3 &                summoner, troll \\
43 &               enigma &      3 &               primordial, void \\
44 &     treant protector &      3 &                   brute, druid \\
45 &                 doom &      4 &                   brute, demon \\
46 &            disraptor &      4 &               brawny, warlock  \\
47 &          void spirit &      4 &                   spirit, void \\
48 &               mirana &      4 &               hunter, vigilant \\
49 &           tidehunter &      4 &               scaled, warrior  \\
50 &            necrophos &      4 &    healer, heartless, warlock  \\
51 &           lone druid &      4 &        druid, savage, summoner \\
52 &                 sven &      4 &          human, knight, scaled \\
53 &                slark &      4 &               assassin, scaled \\
54 &     templar assassin &      4 &       assassin, vigilant, void \\
55 &  keeper of the light &      4 &                    human, mage \\
56 &                  axe &      5 &                  brawny, brute \\
57 &        faceless void &      5 &                 assassin, void \\
58 &            sand king &      5 &                 insect, savage \\
59 &                 lich &      5 &                heartless, mage \\
60 &               medusa &      5 &                 hunter, scaled \\
61 &        dragon knight &      5 &          dragon, human, knight \\
62 &        troll warlord &      5 &                troll, warrior  \\
\bottomrule
\end{tabular}
}
\caption{Heroes power and alliances  structure}
\label{table:aliances}
\end{table}

%%%%%%%%%%%%%%%%%%%%%%%%%%%%%%%%%%%%%%%%%%%%%%
%  RESULT TABLE
%%%%%%%%%%%%%%%%%%%%%%%%%%%%%%%%%%%%%%%%%%%%%%

We provide a solution of the linear integer programming model defined by the system of inequalities \eqref{eq:DUIP}, as a table \ref{table:solution}.

\begin{table}
\resizebox{16cm}{!} {
\begin{tabular}{l| *{8}{p{1.6cm}} | *{3}{ p{1cm}} }
{Hero} &                   &               &                &                &                &                  &                  &                  &  Alliance contribution &  Hero power &   Sum \\
\midrule
broodmother   &  heartless 2 +0.3  &  human 2 +0.3  &  insect 2 +0.3  &  scaled 2 +0.6  &   troll 2 +0.3  &  warlock  2 +0.6  &  warlock  4 +0.6  &                   &                  3.0 &           3 &   6.0 \\
disruptor     &  heartless 2 +0.4  &  human 2 +0.4  &  insect 2 +0.4  &  scaled 2 +0.8  &   troll 2 +0.4  &  warlock  2 +0.8  &  warlock  4 +0.8  &                   &                  4.0 &           4 &   8.0 \\
dragon knight &  heartless 2 +0.5  &  human 2 +0.5  &  insect 2 +0.5  &  knight 2 +1.0  &  scaled 2 +1.0  &     troll 2 +0.5  &  warlock  2 +0.5  &  warlock  4 +0.5  &                  5.0 &           5 &  10.0 \\
lich          &  heartless 2 +0.5  &  human 2 +0.5  &  insect 2 +0.5  &  scaled 2 +1.0  &   troll 2 +0.5  &  warlock  2 +0.5  &  warlock  4 +0.5  &                   &                  4.0 &           5 &   9.0 \\
medusa        &  heartless 2 +0.5  &  human 2 +0.5  &  insect 2 +0.5  &  scaled 2 +1.0  &   troll 2 +0.5  &  warlock  2 +0.5  &  warlock  4 +0.5  &                   &                  4.0 &           5 &   9.0 \\
necrophos     &  heartless 2 +0.4  &  human 2 +0.4  &  insect 2 +0.4  &  scaled 2 +0.8  &   troll 2 +0.4  &  warlock  2 +0.8  &  warlock  4 +0.8  &                   &                  4.0 &           4 &   8.0 \\
sand king     &  heartless 2 +0.5  &  human 2 +0.5  &  insect 2 +0.5  &  scaled 2 +1.0  &   troll 2 +0.5  &  warlock  2 +0.5  &  warlock  4 +0.5  &                   &                  4.0 &           5 &   9.0 \\
sven          &  heartless 2 +0.4  &  human 2 +0.4  &  insect 2 +0.4  &  knight 2 +0.8  &  scaled 2 +0.8  &     troll 2 +0.4  &  warlock  2 +0.4  &  warlock  4 +0.4  &                  4.0 &           4 &   8.0 \\
troll warlord &  heartless 2 +0.5  &  human 2 +0.5  &  insect 2 +0.5  &  scaled 2 +1.0  &   troll 2 +1.0  &  warlock  2 +0.5  &  warlock  4 +0.5  &                   &                  4.5 &           5 &   9.5 \\
witch doctor  &  heartless 2 +0.2  &  human 2 +0.2  &  insect 2 +0.2  &  scaled 2 +0.4  &   troll 2 +0.4  &  warlock  2 +0.4  &  warlock  4 +0.4  &                   &                  2.2 &           2 &   4.2 \\
\bottomrule
\end{tabular}
}
\caption{Optimal team structure for the Dota Underlords game with all the active bonuces}
\label{table:solution}
\end{table}

\section{Conclusion}
\label{SectionConclusion}
In this paper, we demonstrated how a key to winning in a video game can lie in using linear integer math programming.
The initial data and the results of solving the model in the form of a Jupyter-notebook can be found in our open repository \cite{UnderLordsInput}.
We hope that this article will help to attract the attention of young minds to integer programming, discrete optimization methods, and also to the millennium problem P $ \neq? $ NP.
It is important that the mathematical formulation of the problem given by the set of inequalities \eqref{eq:DUIP} can be considered by itself, abstracting from the domain. And in this paper, it is shown that the seemingly NP-hard task, in the ``yes" or ``no" version, is NP-complete.
Thus, this work contributes to the study of NP-complete problems.

\bibliographystyle{unsrt}
\bibliography{references}

\end{document}